\documentclass[12pt]{amsart} 

\usepackage{amssymb}
\usepackage{amsmath}
\usepackage{amsthm}
\usepackage{verbatim}
\usepackage{pst-node}
\usepackage{auto-pst-pdf}
\usepackage{tikz-cd}
\usepackage{url}
\usepackage{color}
\usepackage{mathabx}

\usepackage{tikz}
\usepackage{pgfplots}

\usepackage{graphicx}
\usepackage{calc}

\newcommand{\inv}{^{-1}}

\newcommand{\mcm}[3]{\newcommand{#1}[#2]{{\ensuremath{#3}}}} 
\mcm{\restric}{0}{\upharpoonright}

\numberwithin{equation}{section}

\newtheorem{theorem}[equation]{Theorem}
\newtheorem{lemma}[equation]{Lemma}

\newtheorem*{theorem*}{Theorem}

\theoremstyle{definition}
\newtheorem{defn}[equation]{Definition}

\theoremstyle{remark}

\title{Limits and Colimits in the Category of Pastures}
\author{Steven Creech}
\date{March 2021}

\begin{document}

\maketitle
\begin{abstract}
The category of pastures was introduced by Baker and Lorscheid \cite{baker2018moduli}, \cite{baker2021descartes}. In this paper, we shall show that the category of pastures will have arbitrary limits and colimits indexed by small diagrams. We shall do this by constructing fibered products, fibered coproducts, equalizers, coequalizers, arbitrary products, and arbitrary coproducts (indexed over a small category, i.e. a set). We begin by recalling the definition of pastures and the notion of morphisms of pastures. We then proceed to construct fibered products, fibered coproducts, equalizers, coequalizers, arbitrary products, and arbitrary coproducts.
\end{abstract}

\section{The Category of Pastures}

Let us recall the definitions that define the category of pastures. 

\begin{defn}
A \textit{pasture} is a multiplicative monoid with zero $P$ such that $P^\times=P-\{0\}$ is an abelian group. There is an involution $x\mapsto -x$ which fixes zero, and there is a subset $N_P\subset P^3$ called the \textit{nullset of $P$} (we write $a+b+c=0$ to mean that $(a,b,c)\in N_P$) that satisfies the following properties:
\begin{enumerate}
    \item $N_P$ is invariant under permutations
    \item $N_P$ is invariant under multiplication. That is for $a,b,c,d\in P$ if $a+b+c=0$, then $da+db+dc=0$.
    \item $a+b+0=0$ if and only if $a=-b$.
\end{enumerate}
\end{defn}

We shall view the multiplication of monoids in the following light. We have that $P^\times$ is an abelian group, and $0$ is an absorbing element.

\begin{defn}
A \textit{morphism of pastures} if a multiplicative map $f:P_1\rightarrow P_2$ such that $f(0)=0$, $f(1)=1$, $f(-a)=-f(a)$, and $f(a)+f(b)+f(c)=0$ in $N_{P_2}$ whenever $a+b+c=0$ in $N_{P_1}$ for any $a,b,c\in P_1$.
\end{defn}

In some sense the category of pastures is nicer than that of fields as it has initial and final objects as well as (co)products, fibered (co)products, and (co)equalizers. See \cite{baker2020foundations} for the initial object, final object, product, and coproduct. 

\section{fibered products}
In this section we shall define the fibered product of pastures, and we shall prove that this is the correct notion of the fibered product. Recall that given pastures $P_1$ and $P_2$ with maps $f_1:P_1\rightarrow P$ and $f_2:P_2\rightarrow P$, the fibered product is the pasture $P_1\times_P P_2$ along with maps $\pi_1:P_1\times_P P_2\rightarrow P_1$ and $\pi_2: P_1\times P_2\rightarrow P_2$ that satisfies the universal property that for any pasture $P'$ with morphisms $g_1:P'\rightarrow P_1$ and $g_2:P'\rightarrow P_2$ there is a unique morphsim $g_*:P'\rightarrow P_1\times_P P_2$ such that the following diagram commutes:

\begin{center}

\begin{tikzcd}
P'
\arrow[drr, bend left, "g_1"]
\arrow[ddr, bend right, "g_2"]
\arrow[dr, dotted, "g_*" description] & & \\
& P_1 \times_P P_2\arrow[r, "\pi_1"]\arrow[d, "\pi_2"]& P_1\arrow[d, "f_1"] \\
& P_2\arrow[r, "f_2"]& P
\end{tikzcd}
    
\end{center}

\begin{defn}
The \textit{fibered product of pastures $P_1$ and $P_2$ over the pasture $P$}, denoted $P_1\times_P P_2$, is defined as follows. As a set we have that $P_1\times_P P_2=0\cup \{(a,b): f_1(a)=f_2(b), a\in P_1^\times, b\in P_2^\times\}$. The multiplication is defined component wise with $0$ as an absorbing element and the involution is given by $(a,b)\mapsto (-a,-b)$ for nonzero elements. The additive relations are given by the following rules (and permutations of them):
\begin{itemize}
    \item $0+0+0=0$
    \item $(a,b)+(-a,-b)+0=0$ for $(a,b)\in (P_1\times_P P_2)^\times$
    \item For $(a_1,b_1)+(a_2,b_2)+(a_3,b_3)=0$ if and only if $a_1+a_2+a_3=0$ in $P_1$ and $b_1+b_2+b_3=0$ in $P_2$ for $(a_1,b_1),(a_2,b_2),(a_3,b_3)\in (P_1\times_P P_2)^\times$
\end{itemize}
The map $\pi_1$ is given by $\pi_1(0)=0$ and $\pi_1((a,b))=a$. Similarly, we have that $\pi_2(0)=0$ and $\pi_2((a,b))=b$
\end{defn}

We will now show that this is indeed the correct notion of the fibered product. We shall do this in $3$ lemmas. First we will show that $P_1\times_P P_2$ is indeed a pasture. Second, we will show that $\pi_1$ and $\pi_2$ are morphisms of pastures such that $f_1\circ \pi_1=f_2\circ \pi_2$. Finally, we shall show that $P_1\times_P P_2$ satisfies the universal property of fibered products.  

\begin{lemma}
$P_1\times_P P_2$ as we have defined is indeed a pasture.
\end{lemma}

\begin{proof}
We begin by proving that $P_1\times_P P_2$ is indeed a pasture. We shall first show that $(P_1\times_P P_2)^\times$ is an abelian group. 

We note that by definition we have that $0$ is an absorbing element under multiplication. Furthermore, I claim that $(1,1)$ is the multiplicative identity of $P_1\times_P P_2$. We know that this is indeed an element of $P_1\times_P P_2$ as $f_1$ and $f_2$ are morphisms, so $f_1(1)=1=f_2(1)$. For any other $(a,b)\in P_1\times_P P_2$ we have that $(1,1)(a,b)=(1(a),1(b))=(a,b)$, so $(1,1)$ is indeed the multiplicative identity.

Now let us show that nonzero elements have multiplicative inverses. Now say that $(a,b)\in (P_1\times_P P_2)^\times$; thus, we have that $f_1(a)=f_1(b)$ where $a\in P_1$ and $b\in P_2$. Since $f_1$ and $f_2$ are morphisms,we know that $f_1(a^{-1})=f_1(a)^{-1}$ and $f_2(b^{-1})=f_2(b)^{-1}$. Hence, we have that $f_1(a^{-1})=f_1(a)^{-1}=f_2(b)^{-1}=f_2(b^{-1})$. Thus, we have that $(a^{-1},b^{-1})\in P_1\times_P P_2$ and $(a,b)(a^{-1},b^{-1})=(aa^{-1},bb^{-1})=(1,1)$. 

We need to show that multiplication is closed in $(P_1\times_P P_2)^\times$. Say that $(a_1,b_1),(a_2,b_2)\in(P_1\times_P P_2)^\times$. We wish to show that $(a_1,b_1)(a_2,b_2)=(a_1a_2,b_1b_2)\in (P_1\times_P P_2)^\times$. This will be true if $f_1(a_1a_2)=f_2(b_1b_2)$. This is true as $f_1(a_i)=f_2(b_i)$ for $i=1,2$. Hence, we have that $f_1(a_1a_2)=f_1(a_1)f_1(a_2)=f_2(b_1)f_2(b_2)=f_2(b_1b_2)$. Furthermore, the fact that multiplication is commutative easily follows from the commutativity of multiplication in $P^1$ and $P^2$. This concludes the proof that $(P_1\times_P P_2)^\times$ is an abelian group. 

We know that the involution fixes $0$ and is indeed an involution. However, we must show that if $(a,b)\in(P_1\times_P P_2)^\times$, then $(-a,-b)\in (P_1\times_P P_2)^\times$. This follows since $f_1$ and $f_2$ are morphisms. Thus, we have that $f_1(-a)=-f_1(a)=-f_2(b)=f_2(-b)$. Thus, we have that $(-a,-b)\in (P_1\times_P P_2)^\times$.

Thus, it simply remains to show that the additive relations satisfy the properties of being a nullset. We note by definition, we have that the additive relations will be invariant under permutations which is property (1) of the nullset. Furthermore, as $(a,b)+(-a,-b)+0=0$ we have that property (3) of being a nullset is satisfied. 

To show that property (2) is satisfied if $(a_1,b_1)+(-a_1,-b_1)+0=0$, then for $(a_2,b_2)$ we have that $(a_2,b_2)(a_1,b_1)+(a_2,b_2)(-a_1,-b_2)+0=(a_2a_1,b_2b_1)+(a_2(-a_1),b_2(-b_2))+0=0$. This is true as we know that $a_1+(-a_1)+0=0$ in $P_1$ implies that $a_2a_1+a_2(-a_2)+0=0$ in $P_1$ (as additive relations in $P_1$ are closed under multiplication). Similarly, we know that $b_2b_1+b_2(-b_1)+0=0$ in $P_2$. Now we also wish to show that for $(a_1,b_1)+(a_2,b_2)+(a_3,b_3)=0$, then multiplying by $(a_4,b_4)$ also adds to $0$. This is indeed true as we have that $a_1+a_2+a_3=0$ in $P_1$, so $a_4a_1+a_4a_2+a_4a_3=0$ in $P_1$. Again, we similarly have that $b_4b_1+b_4b_2+b_4b_3=0$. This guarantees that $(a_4,b_4)(a_1,b_1)+(a_4,b_4)(a_2,b_2)+(a_4,b_4)(a_3,b_3)=0$. Hence, we have that the nullset satisfies property (2) which is the final property which we needed to show to prove that $P_1\times_P P_2$ is a pasture. 
\end{proof}

\begin{lemma}
$\pi_1$ and $\pi_2$ are morphisms of pastures such that $f_1\circ \pi_1=f_2\circ \pi_2$. That is the bottom square of the diagram commutes. 
\end{lemma}

\begin{proof}
Now let us show that $\pi_1$ and $\pi_2$ are morphisms of pastures. Without loss of generality, we shall simply show that $\pi_1$ is a morphism of pastures. 

Since $\pi_1(0)=0$ we have this property is satisfied and $\pi_1(1,1)=1$, so $\pi_1$ sends zero and the multiplicative identity to zero and the multiplicative identity, respectively. 

Let us show that $\pi_1$ is a group homomorphism of $(P_1\times_P P_2)^\times$. Say $(a_1,b_1),(a_2,b_2)\in (P_1\times_P P_2)^\times$, then 
\[
\pi_1((a_1,b_1)(a_2,b_2))=\pi_1(a_1a_2,b_1b_2)=a_1a_2=\pi_1(a_1,b_1)\pi_1(a_2,b_2)
\]
That is $\pi_1$ is a group homomorphism in the nonzero elements. 

Now Let us show the involution property, consider $(a,b)\in P_1\times_P P_2$, then $\pi_1(-(a,b))=\pi_1(-a,-b)=-a=-\pi_1(a,b)$; thus, we have the involution is preserved under $\pi_1$. Now it remains to show that $(a_1,b_1)+(a_2,b_2)+(a_3,b_3)=0$ implies that $\pi_1(a_1,b_1)+\pi_1(a_2,b_2)+\pi_1(a_3,b_3)=0$, but this follows immediately from how the additive relations are defined in $P_1\times_P P_2$. Thus, we have that $\pi_1$ and $\pi_2$ are morphisms. 

Now let us show that the diagram commutes, so let $(a,b)\in P_1\times_P P_2$, then $f_1\circ\pi_1(a,b)=f_1(a)$, and $f_2\circ \pi_2(a,b)=f_2(b)$, but we have that the $(a,b)\in P_1\times_P P_2$ are precisely the elements that $f_1(a)=f_2(b)$. Thus, the diagram commutes in the lower square. 
\end{proof}

Now it simply remains to show that $P_1\times_P P_2$ satisfies the universal property of the fibered product. We shall state that it satisfies the condition as a lemma.

\begin{lemma}
If $P'$ is a pasture and $g_1:P'\rightarrow P_1$ and $g_2:P'\rightarrow P_2$ are morphisms such that $f_1\circ g_1=f_2\circ g_1$, then there is a unique morphism $g_*:P'\rightarrow P_1\times_P P_2$ such that the entire diagram commutes.
\end{lemma}
\begin{proof}
Let us define $g_*(x)=(g_1(x),g_2(x))$ for $x\neq 0$ and $g_*(0)=0$. We first note that this is the only possible morphism that would make the diagram commute since $\pi_1\circ g_*(x)=g_1(x)$ and $\pi_2\circ g_*(x)=g_2(x)$. Thus, we have uniqueness of $g_*$. Furthermore, we have that $g_*$ is well-defined, since $f_1\circ g_1=f_2\circ g_2$, so $(g_1(x),g_2(x))\in P_1\times_P P_2$.

It simply remains to show that $g_*$ is a morphism. Since $g_1$ and $g_2$ are morphisms, we have that $g_*(1)=(g_1(1),g_2(1))=(1,1)$, and by definition $g_*(0)=0$, so we have that $g_*$ sends zero and the multiplicative unit to zero and the multiplicative unit, respectively. Now let us show that $g_*$ is a group homomorphism on the nonzero elements. So let $x,y\in P'$, then 
\[
g_*(xy)=(g_1(xy),g_2(xy))=(g_1(x)g_1(y),g_2(x)g_2(y))
\]
\[
=(g_1(x),g_2(x))(g_1(y),g_2(y))=g_*(x)g_*(y)
\]
So we have that $g_*$ is a group homomorphism on the nonzero elements.

Now to show that the involution property holds, we note that:
\[
g_*(-x)=(g_1(-x),g_2(-x))=(-g_1(x),-g_2(x))=-(g_1(x),g_2(x))=-g_*(x)
\]

It now remains to show that $g_*$ preserves the additive relations. Say $x+y+z=0$ in $P'$. We wish to show that $g_*(x)+g_*(y)+g_*(z)=0$, but
\[
g_*(x)+g_*(y)+g_*(z)=(g_1(x),g_2(x))+(g_1(y),g_2(y))+(g_1(z),g_2(z))
\]
Since $g_1$ is a morphism, we know that $g_1(x)+g_1(y)+g_1(z)=0$. Similarly, $g_2(x)+g_2(y)+g_2(z)=0$; thus, by how the additive relations are defined in $P_1\times_P P_2$ we have that $g_*(x)+g_*(y)+g_*(z)=0$. Thus, $g_*$ respects the additive relations, and $g_*$ is a morphism of pastures.
\end{proof}

We will summarize the following three lemmas as a theorem.

\begin{theorem}
The category of pastures has fibered products.
\end{theorem}

\section{fibered coproducts}

In this section we shall define the fibered coproduct of pastures, and we shall prove that this is the correct notion of the fibered coproduct. Recall that given a pastures $P_1$ and $P_2$ with maps $f_1:P\rightarrow P_1$ and $f_2:P\rightarrow P_2$, then the fibered coproduct is the pasture $P_1\otimes_P P_2$ along with maps $i_1:P_1\rightarrow P_1\otimes_P P_2$ and $i_2:P_2\rightarrow P_1\otimes_P P_2$ that satisfies the universal property that for any pasture $P'$ with morphisms $g_1:P_1\rightarrow P'$ and $g_2:P_2\rightarrow P'$, there is a unique morphism $g_*:P_1\otimes_P P_2\rightarrow P'$ such that the following diagram commutes:

\begin{center}

\begin{tikzcd}
P' & & \\
& P_1 \otimes_P P_2\arrow[ul,dotted,"g_*"]& P_1 \arrow[l,"i_1"]\arrow[llu,bend right, "g_1"] \\
& P_2\arrow[uul,bend left,"g_2"]\arrow[u, "i_2"]& P\arrow[l,"f_2"]\arrow[u,"f_1"]
\end{tikzcd}
    
\end{center}

\begin{defn}
The \textit{fibered coproduct of pastures $P_1$ and $P_2$ over the pasture $P$}, denoted $P_1\otimes_P P_2$, is defined as follows. As a set we have that $P_1\otimes_P P_2=P_1\times P_2/\sim$ where $P_1\times P_2$ is the Cartesian product, and the equivalence relation $\sim$ is defined as follows. $(a_1,b_1)\sim (a_2,b_2)$ if and only if one of the following hold: 
\begin{itemize}
    \item $a_1b_1=0$ or $a_2b_2=0$
    \item $(f_1(x)a,b)\sim (a,f_2(x)b)$ for $x\in P^\times$.
\end{itemize}
We can restate the second condition as $(a,b)\sim (c,d)$ if and only if $(c,d)=(f_1(x)^{-1}a,f_2(x)b)$ for some $x\in P^\times$. Denoting the equivalence class of $(a,b)$ as $[(a,b)]$. We have that $P_1\otimes_P P_2$ is a pasture where we have that $0=[(0,0)]$, $1=[(1,1)]$, $-[(a,b)]=[(-a,b)]=[(a,-b)]$. The multiplicative structure on $P_1\otimes_P P_2$ is given by $[(0,0)][(a,b)]=[(0,0)]$, and $[(a,b)][(c,d)]=[(ac,bd)]$ for $[(a,b)],[(c,d)]\in P_1\otimes_P P_2$. The additive relations of $P_1\otimes_P P_2$ are the given by:
\begin{itemize}
    \item $[(a,y)]+[(b,y)]+[(c,y)]=0$ for $y\in P_2$ and $a,b,c\in P_1$ with $a+b+c=0$ in $P_1$
    \item $[(x,a)]+[(x,b)]+[(x,c)]=0$ for $x\in P_1$ and $a,b,c\in P_2$ with $a+b+c=0$ in $P_2$.
\end{itemize}
We then define the maps $i_1$ and $i_2$ by $i_1: P_1\rightarrow P_1\otimes_P P_2$ is given by $0\mapsto [(0,0)]$ and $x\mapsto [(x,1)]$, and $i_2: P_2\rightarrow P_1\otimes_P P_2$ is given by $0\mapsto [(0,0)]$ and $y\mapsto [(1,y)]$. 
\end{defn}

Now we will begin by showing that $\sim$ is indeed an equivalence relation. Then we shall show that $P_1\otimes_P P_2$ is indeed a pasture. We will then show that $i_1$ and $i_2$ are homomorphisms of pastures that make the correct diagram commute, and conclude by showing that $P_1\otimes_P P_2$ satisfies the universal property of fibered coproducts.

\begin{lemma}
The equivalence relation defined in the definition of $P_1\otimes_P P_2$ is actually an equivalence relation.
\end{lemma}

\begin{proof}
We first note that the first condition says that any pair that has some $0$ is equivalent to the pair $(0,0)$. That is $[(0,0)]=\{(a,b): \text{one of $a$ or $b$ is $0$}\}$.

Let us now simply focus on the case of $[(a,b)]$ where $a,b\neq 0$.

Reflexive: We note that $f_1(1)=f_2(1)=1$, so $(a,b)=(f_1(1)a,b)\sim (a,f_2(1)b)=(a,b)$ that is $\sim$ is reflexive.

Symmetric: Using the equivalent statement of the second condition, say we have that $(a,b)\sim (c,d)$ that is $(c,d)=(f_1(x)\inv a,f_2(x)b)$. We wish to show that $(f_1(x)\inv a,f_2(x)b)\sim (a,b)$, but $(a,b)=(f_1(x\inv)\inv f_1(x)\inv, f_2(x\inv)f_2(x)b)$; thus, we we have that $(c,d)\sim (a,b)$. Hence, it is symmetric.

Transitive: Say $(a,b)\sim (c,d)$ and $(c,d)\sim (e,f)$. Thus, $(c,d)=((f_1(x))\inv a,f_2(x)b)$ for some $x\in P^\times$ and $(e,f)=(f_1(y)\inv c,f_2(y)d)$ for some $y\in P^\times$. Thus, we get that:
\[
(e,f)=((f_1(y))\inv c,f_2(y)d)=((f_1(y))\inv (f_1(x))\inv a,f_2(y)f_2(x)b)=((f_1(xy))\inv a,f_2(xy)b)
\]
That is $(a,b)\sim (e,f)$. Thus, $\sim$ is an equivalence relation.
\end{proof}

\begin{lemma}
The pasture $P_1\otimes_P P_2$ as defined is indeed a pasture.
\end{lemma}

\begin{proof}
The first step that we shall take is to show that multiplication is well-defined in $P_1\otimes_P P_2$. 

Since for $(a,b)\sim (a',b')$ and $(c,d)\sim (c',d')$, we have that:
\[
(a,b)(c,d)=(ac,bd)
\]
and 
\[
(a',b')(c',d')=(a'c',b'd')
\]
However, $a'=(f_1(x))\inv a$ and $b'=f_2(x)b$ for some $x\in P^\times$ and $c'=(f_1(y))\inv c$ and $d'=f_2(y)d$ for some $y\in P^\times$; thus,
\[
(a'c',b'd')=((f_1(x))\inv a(f_1(y))\inv c, f_2(x)bf_2(y)d)=((f_1(xy))\inv ac,f_2(xy)bd)
\]
Thus, we have that $(a'c',b'd')\sim (ac,bd)$. Thus, the multiplication is well-defined.

Now by definition we have that $[(0,0)][(a,b)]=[(0,0)]$, so zero is an absorbing element. Let us now show that $(P_1\otimes_P P_2)^\times$ is an abelian group. We note that by showing that multiplication is well-defined, we have that $(P_1\otimes_P P_2)^\times$ is closed under multiplication. 

We now claim that the multiplicative identity is $[(1,1)]$. This is true since we have that $[(1,1)][(a,b)]=[(1(a),1(b))]=[(a,b)]$. We have inverses as $[(a,b)]\inv=[(a\inv,b\inv)]$ since $[(a,b)][(a,b)]\inv=[(a,b)][(a\inv,b\inv)]=[(1,1)]$. This concludes showing that $(P_1\otimes_P P_2)^\times$ is an abelian group.

Let us now show that the involution is indeed an involution. This is fairly obvious from the definition that $-[(a,b)]=[(-a,b)]=[(a,-b)]$. However, we should show that $[(-a,b)]\sim [(a,-b)]$. This follows from the fact that in a pasture we have that $-1\cdot a=-a$. This is easily seen by seeing that $1+-1+0=0$, so multiplying by $a$ gives $a+(-1)a+0=0$, and this implies that $-a=-1\cdot a$. Thus, we observe that $[(-a,b)]\sim[(a,-b)]$ since $(a,-b)=(f_1(-1)\inv(-a),f_2(-1)b)$.

Now we will show that the additive relations satisfy the properties of the null set. Without loss of generality, let us just focus on the additive relations of the form $[(a,y)]+[(b,y)]+[(c,y)]=0$ for $y\in P_2$ and $a,b,c\in P_1$ with $a+b+c=0$. We note that if we permute these element since any permutation of $a+b+c=0$ also is an additive relation, we have that any permutation of $[(a,y)]+[(b,y)]+[(c,y)]=0$ is also an additive relation. Thus, we have that property (1) of the null set is satisfied. 

Say that $[(a,y)]+[(b,y)]+[(c,y)]=0$ and $[(d,z)]\in P_1\otimes_P P_2$. Then we have that $[(d,z)][(a,y)]+[(d,z)][(b,y)]+[(d,z)][(c,y)]=[(da,zy)]+[(db,zy)]+[(dc,zy)]=0$ since we have that $a+b+c=0$ in $P_1$, so $da+db+dc=0$ in $P_1$ and we have that $zy\in P_2$. Hence we have that property (2) of the null set is satisfied. 

Finally, say that we have that $[(a,y)]+[(b,y)]+[(0,0)]=0$, then this is an additive relation if and only if $a+b+0=0$ in $P_1$. However, this occurs if and only if $a=-b$. We note that $-[(b,y)]=[(-b,y)]=[(a,y)]$. Thus, property (3) of the null set is satisfied. 

Thus, we have that $P_1\otimes_P P_2$ satisfies all the axioms of a pasture.
\end{proof}

\begin{lemma}
The functions $i_1$ and $i_2$ are morphisms of pastures such that $i_1\circ f_1=i_2\circ f_2$.
\end{lemma}

\begin{proof}
Without loss of generality, let us show that $i_1$ is a morphism of pastures. Let us first show that $i_1$ is a homomorphism of the abelian groups $P_1^\times\rightarrow (P_1\otimes_P P_2)^\times$. Say that $x,y\in P_1^\times$, then we note that
\[
i_1(xy)=[(xy,1)]=[(x,1)][(y,1)]=i_1(x)i_1(y)
\]
so $i_1$ is a homomorphism of groups. Furthermore, we note that $i_1(0)=[(0,1)]=[(0,0)]$, and $i_1(1)=[(1,1)]$. Thus, $i_1$ sends the zero and multiplicative identity to zero and the multiplicative identity respectively. 

Now let us show that $i_1$ preserves the involution. This is clear since we have that $i_1(-x)=[(-x,1)]=-[(x,1)]=-i_1(x)$. 

Now it simply remains to show that the additive relations are preserved that is if $a+b+c=0$ in $P_1$, then $i_1(a)+i_1(b)+i_1(c)=0$ in $P_1\otimes_P P_2$. This is fairly clear since this just says that $[(a,1)]+[(b,1)]+[(c,1)]=0$ in $P_1\otimes_P P_2$ whenever $a+b+c=0$ in $P_1$, but this is exactly how the additive relations are defined. Hence we have that $i_1$ is a morphism.

Now let us show that $i_1\circ f_1=i_2\circ f_2$. So consider $x\in P$. Then we have that $i_1(f_1(x))=[(f_1(x),1)]$ and $i_2(f_2(x))=[(1,f_2(x))]$. However, we have that $[(f_1(x),1)]\sim [(1,f_2(x))]$, so we have that $i_1\circ f_1=i_2\circ f_2$.
\end{proof}

\begin{lemma}
If $P'$ is a pasture and $g_1:P_1\rightarrow P'$ and $g_2:P_2\rightarrow P'$ are morphisms such that $g_1\circ f_1=g_2\circ f_2$, then there is a unique morphism $g_*:P_1\otimes_P P_2\rightarrow P'$. That is $P_1\otimes_P P_2$ satisfies the universal property of fibered coproducts.
\end{lemma}

\begin{proof}
Let us define $g_*([(a,b)])=g_1(a)g_2(b)$. Let us show that $g_*$ is well-defined. Let $(a,b)\sim (a',b')$. So $g_*((a,b))=g_1(a)g_2(b)$, and $g_*((a',b'))=g_1(a')g_2(b')$. Since $(a,b)\sim (a',b')$, we have that $a'=(f_1(x))\inv a$ and $b'=f_2(x)b$ for some $x\in P^\times$. Thus, 
\[
g_1(a')g_2(b')=g_1((f_1(x))\inv a)g_2(f_2(x)b)=g_1((f_1(x))\inv)g_1(a)g_2(f_2(x))g_2(b)
\]
\[
= g_1(a)g_2(b)(g_1(f_1(x)))\inv g_2(f_2(x))
\]
However, since we have $g_1\circ f_1=g_2\circ f_2$, we must have $(g_1(f_1(x)))\inv g_2(f_2(x))=1$. Hence $g_1(a)g_2(b)=g_1(a')g_2(b')$ that is $g_*$ is a well-defined function.

Now let us show that $g_*$ is a morphism of pastures. First let us show that $g_*$ is a group homomorphism of the nonzero elements. This is clear since
\[
g_*([(a,b)][(c,d)]=g_*([(ac,bd)])=g_1(ac)g_2(bd)=g_1(a)g_2(b)g_1(c)g_2(d)
\]
\[
=g_*([(a,b)])g_*([(c,d)])
\]
Thus, we have that $g_*$ is a group homomorphism of the nonzero elements. Now we have that $g_*$ preserves the involution since 
\[
g_*(-[(a,b)])=g_*([(-a,b)])=g_1(-a)g_2(b)=-g_1(a)g_2(b)=-g([(a,b)])
\]
Now it remains to show that $g_*$ preserves the additive relations. Without loss of generality, say that we have that $[(a,y)]+[(b,y)]+[(c,y)]=0$, then:
\[
g_*([(a,y)])+g_*([(b,y)])+g_*([(c,y)])
\]
\[
=g_1(a)g_2(y)+g_1(b)g_2(y)+g_1(c)g_2(y)=0
\]
Since $a+b+c=0$ and $g_1$ is a morphism $g_1(a)+g_1(b)+g_1(c)=0$; thus, the multiplication by $g_2(y)$ still preserves the additive relation. Thus, we have that $g_*$ preserves the additive relations, and $g_*$ is a morphism.

Now let us show that $g_*$ as it has been defined will make the entire diagram commute. Let $x\in P_1$, we wish to show that $g_*\circ i_1(x)=g_1(x)$. This is the case since $g_*\circ i_1(x)=g_*([(x,1)])=g_1(x)g_2(1)=g_1(x)$. Similarly, we have that $g_*\circ i_2=g_2$. Thus, the diagram will commute.

It now only remains to show that the morphism $g_*$ is the unique morphism from $P_1\otimes_P P_2\rightarrow P'$ that makes the diagram commute. We note that for the diagram to commute, we require that $g_*([(a,1)])=g_1(a)$ and $g_*([(1,b)])=g_2(b)$. Observe:
\[
g_*([(a,b)])=g_*([(a,1)][(1,b)])=g_*([(a,1)])g_*([(1,b)])=g_1(a)g_2(b)
\]
That is $g_*$ is the only such morphism that will make the diagram commute. Showing that we have that $g_*$ is unique.
\end{proof}

We summarize our results as the following theorem

\begin{theorem}
The category of pastures has fibered coproducts.
\end{theorem}

\section{equalizers}
In this section, we shall show that equalizer of two functions $f$ and $g$ exists in the category of pastures. We shall then show that the notion of equalizer as we define it is indeed the correct notion. Recall that given pastures $P_1, P_2$ and maps $f,g:P_1\rightarrow P_2$. The equalizer of $f$ and $g$ is a pasture $Q$ along with a map $q:Q\rightarrow P_1$ such that for any other pasture $Q'$ with a map $q':Q'\rightarrow P_1$ there is a unique map $u:Q'\rightarrow Q$ such that the following diagram commutes:

\begin{center}
\begin{tikzcd}
Q\arrow[r,"q"] & P_1 \arrow[r,shift left=2.5,"f"]\arrow[r,shift right,"g"] & P_2\\
Q'\arrow[u,dotted,"u"]\arrow[ur,"q'"]
\end{tikzcd}
\end{center}

\begin{defn}
Given pastures $P_1$ and $P_2$ and maps $f,g:P_1\rightarrow P_2$ the equalizer of $f$ and $g$ will be given by the subpasture $Q=\{x\in P_1: f(x)=g(x)\}$. Thus, we have that the multiplication is just given by the multiplication in $P_1$. The additive relations are given by $a+b+c=0$ in $Q$ if and only if $a+b+c=0$ in $P_1$ for $a,b,c\in Q$. The map $q:Q\rightarrow P_1$ is just given by inclusion. 
\end{defn}

\begin{lemma}
$Q$ is a pasture, namely it is a subpasture of $P_1$, and the inclusion $q:Q\rightarrow P_1$ is a morphism of pastures such that $f\circ q=g\circ q$.  
\end{lemma}

\begin{proof}
We begin by noting that $0,1\in Q$ since we have that $f,g$ are morphisms we must have that $f(0)=g(0)=0$ and $f(1)=g(1)=1$. To show that the nonzero elements of $Q$ form an abelian group, let us show that if $x,y\in Q$, then $xy\in Q$. This simply follows from $f,g$ being morphisms, so we have that $f(xy)=f(x)f(y)=g(x)g(y)=g(xy)$. Now let us show that we have inverses in $Q$, so if $x\in Q$, let us show that $x\inv\in Q$. Again this follows from using the fact that $f(x)=g(x)$, and $f$ and $g$ are morphisms we have that $f(x\inv)=f(x)\inv=g(x)\inv=g(x\inv)$. Thus, we have that $x\inv\in Q$.  This assures that $Q^\times$ is an abelian group. 

Now we have that $Q$ is closed under taking the involution. That is for $x\in Q$ we have that $-x\in Q$. This is clear since $f(-x)=-f(x)=-g(x)=g(-x)$. 

Now the last thing that remains to show is the additive relations satisfy the properties of a nullset. We note that since $a+b+c=0$ in $Q$ is true if and only if $a+b+c=0$ in $P_1$. Thus, if we permute $a,b,c$ in $P_1$, then it will still add to $0$ in $P_1$, so it will add to $0$ in $Q$. Thus, $N_Q$ will satisfy property (1) of the nullset. Now if $a+b+c=0$ in $Q$ and $d\in Q$, then we know that $a+b+c=0$ in $P_1$, so as $P_1$ is a pasture and $d\in P_1$, we have that $da+db+dc=0$ in $P_1$, but this gives us that $da+db+dc=0$ in $Q$. Thus, we have that $N_Q$ satisfies property (2) of being a nullset. Finally, we know that if $a+b+0=0$ in $Q$, then $a+b+0=0$ in $P_1$, but as $P_1$ is a pasture this implies that $a=-b$. Thus, the $N_Q$ satisfies property (3) of being a nullset. Thus, we have that $Q$ is indeed a pasture, namely it is a subpasture of $P_1$.

Since $Q$ is a subpasture of $P_1$, we have that the inclusion $q: Q\rightarrow P_1$ is a morphism. Furthermore, we have that for $x\in Q$ that $f(q(x))=f(x)=g(x)=g(q(x))$, so we have that $f\circ q=g\circ q$.
\end{proof}

\begin{lemma}
If $Q'$ is a pasture and $q':Q'\rightarrow P_1$ is a morphism such that $f\circ q'=g\circ q'$, then there is a unique map $u:Q'\rightarrow Q$ such that $q\circ u=q'$.
\end{lemma}

\begin{proof}
We have that $Q$ is a subpasture of $P_1$. Thus, we will show that the image of $Q'$ under $q'$ is contained in $Q$. To see this we note that since $f\circ q'=g\circ q'$, so for $x\in Q'$ we have that $f(q'(x))=g(q'(x))$, so we must have that $q'(x)\in Q$. Now let us define the map $u:Q'\rightarrow Q$ by $u(x)=q'(x)$. This is clearly a morphism of pastures as $q'(x)$ is a morphism; furthermore, we have that $q(u(x))=q'(x)$. 

Now let us show that any such morphism is unique. That is if $u':Q'\rightarrow Q$ is a map such that $q\circ u'=q'$ then $u=u'$. This is quite clear since we have that $q$ is just an inclusion map, so we must have that $q'=q\circ u'=u'$, but we have that $q'=u$. Thus, we indeed have that the map $u$ is unique, and $Q$ satisfies the universal property of equalizers. 
\end{proof}

We summarize our results in the following theorem.

\begin{theorem}
The category of pastures has equalizers. 
\end{theorem}

\section{coequalizers}
In this section, we shall show that coequalizers exist in the category of pastures. Recall that given pastures $P_1$, $P_2$ and maps $f,g:F_1\rightarrow F_2$ the coequalizer of $f$ and $g$ is a pasture $Q$ along with a map $q:P_2\rightarrow Q$ such that for any other pasture $Q'$ with a map $q': P_2\rightarrow Q'$ there is a unique map $u:Q\rightarrow Q'$ that makes the following diagram commute:

\begin{center}
\begin{tikzcd}
 P_1 \arrow[r,shift left=2.5,"f"]\arrow[r,shift right,"g"] & P_2\arrow[r,"q"]\arrow[dr,"q'"]&Q\arrow[d,dotted,"u"]\\
 & & Q'
\end{tikzcd}
\end{center}

\begin{defn}
Given pastures $P_1$ and $P_2$ and maps $f,g:P_1\rightarrow P_2$ the coequalizer of $f$ and $g$, the coequalizer of $f$ and $g$ is given by $Q$. As a set we have that $Q=\{0\}\cup P_2^\times/\sim$ where the relation $\sim$ is given by $x\sim y$ if $xy\inv=f(z)g(z)\inv$ for some $z\in P_1$. We denote $[x]$ to be the equivalence class of $x$, and for simplicity we will let $[0]=0$. In that case we have that the multiplicative relations are given by $[x][y]=[xy]$ for $[x],[y]\in Q^\times$. The involution of $x$ is given by $-[x]=[-x]$, and the involution will fix $0$. Then the additive relations will be given as follows. If $a+b+c=0$ in $P_2$, then we have that $[a]+[b]+[c]=0$ in $Q$. Finally, we have that the map $q:P_2\rightarrow Q$ will be given by $q(x)=[x]$.
\end{defn}

Now we will begin by showing that $\sim$ is actually an equivalence relation. Then we will show that $Q$ is a pasture, and $q:P_2\rightarrow Q$ is a morphism of pastures. Finally, we shall show that $Q$ satisfies the universal property of being a coequalizer of $f$ and $g$.

\begin{lemma}
The equivalence relation on $P_2^\times$ given by $x\sim y$ if $xy\inv=f(z)g(z)\inv$ for some $z\in P_2^\times$ is indeed an equivalence relation. 
\end{lemma}

\begin{proof}
To see that $\sim$ is reflexive. This is true since $xx\inv=1=f(1)g(1)\inv$. Now to show that $\sim$ is symmetric, say that $x\sim y$, so there is some $z\in P_1$ such that $xy\inv=f(z)g(z)\inv$. Thus, we have that $yx\inv=f(z\inv)g(z\inv)\inv$, so $y\sim x$. Finally to show that $\sim$ is transitive, say that $x\sim y$ and $y\sim z$. Thus, we have that $xy\inv=f(w_1)g(w_1)\inv$ and $yz\inv=f(w_2)g(w_2)\inv$ for some $w_1,w_2\in P_1$. Thus, we have that $xz\inv=f(w_1)g(w_1)\inv f(w_2)g(w_2)\inv=f(w_1w_2)g(w_1w_2)\inv$. Thus, we have that $\sim$ is transitive and an equivalence relation. 
\end{proof}

\begin{lemma}
The coequalizer $Q$ as defined is a pasture, and the map $q:P_2\rightarrow Q$ is a morphism of pastures. 
\end{lemma}

\begin{proof}
We first note that $[0]$ is the absorbing element under multiplication. Now we will show that $Q^\times$ is an abelian group. The main thing that we will show is that multiplication is well-defined. Thus, for $x,y,x',y'\in P_2$ say that $x\sim x'$ and $y\sim y'$ let us show that $xy\sim x'y'$. Since $x\sim x'$ we have that $x(x')\inv=f(z)g(z)\inv$ for some $z\in P_1$. Similarly, we have that $y(y')\inv=f(z')g(z')\inv$ for some $z'\in P_1$. Thus, we have that $xy(x'y')\inv=f(z)g(z)\inv f(z')g(z')\inv=f(zz')g(zz')\inv$. Thus, we have that $xy\sim x'y'$ that is multiplication is well-defined. 

Now that we have that multiplication is well-defined. Let us show that $Q^\times$ is an abelian group. The identity element will be $[1]$ since $[1][x]=[1x]=[x]$ for any $[x]\in Q$. We also have multiplication is closed since for $[x],[y]\in Q^\times$ we have that $[x][y]=[xy]\in Q^\times$. Now let us show that we have inverses. This is also clear for $[x]\in Q^\times$ we have that $[x\inv]\in Q\times$, and $[x][x\inv]=[xx\inv]=[1]$. Thus, we have that $Q^\times$ is an abelian group.

Now we note that the involution is given by $-[x]=[-x]$ for $[x]\in Q$. Let us show that this is well-defined. To do this we shall show that for $x\sim y$ for $x,y\in P_2$, then we have that $-x\sim -y$. This is clear since $x\sim y$ we have that $xy\inv=f(z)g(z)\inv$ for some $z\in P_1$. We wish to show that $-x\sim -y$, but we have that $(-x)(-y)\inv=(-1)x(-1)y\inv=xy\inv=f(z)g(z)\inv$. Thus, we know that the involution is well-defined, and it is clearly an involution. 

Finally, we wish to show that the nullset $N_Q$ satisfies the additive relations. This is relatively straight forward to see since $[a]+[b]+[c]=0$ in $Q$ if there is some $a'\in[a]$, $b'\in[b]$, and $c'\in[c]$ such that $a'+b'+c'=0$ in $P_2$. Thus, as $P_2$ is a pasture, we know that any permutation of the $a',b',c'$ still adds to zero in $P_2$, so any permutation of $[a],[b],[c]$ still adds to zero in $Q$. That is $N_Q$ satisfies property (1) of being a nullset.

Now say that $[a]+[b]+[c]=0$ in $Q$ and $[d]\in Q$. Thus, we have that there is some $a'\in[a]$, $b'\in[b]$ and $c'\in[c]$ such that $a'+b'+c'=0$ in $P_2$. However, as $P_2$ is a pasture, we must have that $da'+db'+dc'=0$ in $P_2$, so we get that $[d][a]+[d][b]+[d][c]=[da]+[db]+[dc]=0$ in $Q$. Thus, we have that $N_Q$ satisfies property (2) of being a nullset.

Now say that $[a]+[b]+0=0$ in $Q$, so we have that there is some $a'\in [a]$ and $b'\in [b]$ such that $a'+b'+0=0$ in $P_2$, but as $P_2$ is a pasture this tells us that $a'=-b'$. Hence we have that $-[a]=-[a']=[-a']=[-b']=-[b']=-[b]$. That is we have that $N_Q$ satisfies property (3) of being a nullset. This concludes the proof that $Q$ is a pasture. 

Now let us show that $q:P_2\rightarrow Q$ given by $q(x)=[x]$ is a morphism of pastures. We see that $q(0)=[0]=0$ and $q(1)=[1]$. Now let us show that $q:P_2^\times\rightarrow Q^\times$ is an abelian group homomorphism. To see this we simply note that $q(xy)=[xy]=[x][y]=q(x)q(y)$. Thus, is is an abelian group homomorphism. Now let us show that $q$ respects the involution. This is also clear since $q(-x)=[-x]=-[x]=-q(x)$. Finally, to show that $q$ respects additive relations. We note by definition that if $a+b+c=0$ in $P_2$, then we have that $q(a)+q(b)+q(c)=[a]+[b]+[c]=0$ in $Q$. Thus, we have that $q$ is indeed a morphism of pastures.
\end{proof}

Now it remains to show that $Q$ satisfies the universal property of coequalizers. That is if $Q'$ is another pasture and $q':P_2\rightarrow Q'$ is a morphism of pastures such that $q'\circ f=q'\circ g$, then there is a unique morphism $u:Q\rightarrow Q'$ such that $u\circ q=q'$.

\begin{lemma}
If $Q'$ is a pasture and $q':P_2\rightarrow Q'$ is a morphism such that $q'\circ f=q'\circ g$, then there exists a unique morphism $u:Q\rightarrow Q'$ such that $u\circ q=q'$. That is $Q$ satisfies the universal property of being a coequalizer of $f$ and $g$. 
\end{lemma}

\begin{proof}

Let us define the morphism $u:Q\rightarrow Q'$ as follows $u([x])=q'(x)$. We note that by defining $u$ in this way it is clear that $u(q(x))=u([x])=q'(x)$, so we have that the diagram will commute. Thus, we simply need to show that $u$ is well-defined, and is indeed a morphism of pastures. Then we will show that $u$ is unique. 

Let us first by showing that $u$ is well-defined. This amounts to showing that for $x\sim y$ for $x,y\in P_2^\times$, we have that $q'(x)=q'(y)$. We note that we restrict to the case of $P_2^\times$ since we have that $[0]=0$, and $u(0)=q'(0)=0$. Since we have that $x\sim y$ we know that $xy\inv=f(z)g(z)\inv$ for some $z\in P_1^
\times$. We shall show that $q'(xy\inv)=1$ as this will imply that $q'(x)=q'(y)$. Now let us consider $q'(xy\inv)=q'(f(z)g(z)\inv)=q'(f(z))q'(g(z)\inv)=q'(f(z))q'(g(z))\inv=1$ since we know that $q'(f(z))=q'(g(z))$. Thus, we have that our function $u$ is well-defined. 

Now let us show that $u$ is in fact a morphism of pastures. First we note that $u([1])=q'(1)=1$, and $u([0])=q'(0)=0$. Now let us show that $u$ is an abelian group homomorphism of $Q^\times$ and $(Q')^\times$. This is clear since we have that $u([x][y])=q'(xy)=q'(x)q'(y)=u([x])u([y])$. Now let us show that $u$ respects the additive relations. This is also clear. Since if $[a]+[b]+[c]=0$ in $Q$, then we must have that there is some $a'\in [a]$, $b'\in [b]$, and $c'\in[c]$ such that $a'+b'+c'=0$ in $P_2$. However, as $q'$ is a morphism of pastures, we have that $q'(a')+q'(b')+q'(c')=0$ in $Q'$. Thus, we have that $u([a])+u([b])+u([c])=u([a'])+u([b'])+u([c'])=q'(a')+q'(b')+q'(c')=0$. This tells us that $u$ preserves the additive relations, and is hence a morphism of pastures. 

The final thing which remains to show is that $u$ is the unique morphism with this property. Thus say that $u':Q\rightarrow Q'$ is a morphism such that $u'\circ q=q'$. Let us show that $u'=u$. Since $u'$ is a morphism we have that $u'([0])=0$. Thus, let us show that $u'([x])=u([x])$ for $[x]\in Q^\times$. This is clear, since we know that $u([x])=q'(x)$. Furthermore, since we have that $u'\circ q=q'$. We must have that $u'(q(x))=q'(x)$, so it is the case that $u'(q(x))=u'([x])=q'(x)=u(x)$. Thus, we have that $u$ is unique, and $Q$ satisfies the universal property of being the coequalizer of $f$ and $g$. 
\end{proof}

We summarize our results as the following theorem.

\begin{theorem}
The category of pastures has coequalizers. 
\end{theorem}

\section{Arbitrary Products}

In this section, we shall show that arbitrary products exist in the category of pastures. Thus, if $I$ is some indexing set, and for each $i\in I$ we have some pasture $P_i$. Then the product of the $P_i$ is some pasture $\prod_{i\in I}P_i$ along with maps $\pi_j:\prod_{i\in I}P_i\rightarrow P_j$ for $j\in I$. Such that for any pasture $P'$ with maps $f_i:P'\rightarrow P_i$ for each $i\in I$, then there is a unique morphism $f_*:P'\rightarrow \prod_{i\in I}P_i$ such that $\pi_i\circ f_*=f_i$ for all $i\in I$. 
\begin{defn}
Given some indexing set $I$, and for each $i\in I$ we have pastures $P_i$, then we have the product $\prod_{i\in I}P_i$ is a pasture. As a set, we have that $\prod_{i\in I}P_i=\{0\}\cup \prod_{i\in I} P_i^\times$. Now we define the multiplication on $\prod_{i\in I}P_i$ as follows. We let $0$ be an absorbing element, and for $(x_i)_{i\in I}, (y_i)_{i\in I}\in (\prod_{i\in I}P_i)^\times$, we define the product $(x_i)_{i\in I}(y_i)_{i\in I}=(x_iy_i)_{i\in I}$. Then the involution we will have that the involution fixes $0$, and for $(x_i)_{i\in I}\in (\prod_{i\in I}P_i)^\times$, we have that $-(x_i)_{i\in I}=(-x_i)_{i\in I}$. Now to define the additive relations, we say $(a_i)_{i\in I}+(b_i)_{i\in I}+(c_i)_{i\in I}=0$ if $a_i+b_i+c_i=0$ in $P_i$ for each $i\in I$. For simplicity in this definition, we let $0=(0)_{i\in I}$. Now we define the maps $\pi_j:\prod_{i\in I}P_i\rightarrow P_j$ by $\pi_j((x_i)_{i\in I})=x_j$. 
\end{defn}

We shall begin by showing that $\prod_{i\in I}P_i$ is indeed a pasture. We will then go ahead to show that the $\pi_j$ are morphisms of pastures. Finally we will show that $\prod_{i\in I}P_i$ satisfies the universal property of being a product. 

\begin{lemma}
$\prod_{i\in I}P_i$ is a pasture. 
\end{lemma}

\begin{proof}
Let us show that $\prod_{i\in I}P_i$ is a pasture. We note by definition we have that $0$ is an absorbing element under multiplication. Now let us show that $(\prod_{i\in I}P_i)^\times$ is an abelian group. Let us begin by showing that the multiplication is closed. We note for $(x_i)_{i\in I}, (y_i)_{i\in I}\in (\prod_{i\in I}P_i)^\times$ we have that for each $i\in I$ that $x_i,y_i\in P_i^\times$. Thus, when we take the product $(x_i)_{i\in I}(y_i)_{i\in I}=(x_iy_i)_{i\in I}$ Now since for each $i\in I$, we have that $x_i,y_i\in P_i^\times$, so $x_iy_i\in P_i^\times$. Thus, we have that $(x_iy_i)_{i\in I}\in \prod_{i\in I}P_i^\times=(\prod_{i\in I}P_i)^\times$, so $(\prod_{i\in I}P_i)^\times$ is closed under multiplication. 

Now let us show that we have an identity element. We note that for each $i\in I$, $P_i$ is a unit, so there is some identity $1_i\in P_i$. I claim that the identity element of $(\prod_{i\in I}P_i)^\times$ is $(1_i)_{i\in I}$. This is indeed true since we have that $(1_i)_{i\in I}(x_i)_{i\in I}=(1_ix_i)_{i\in I}=(x_i)_{i\in I}$. Thus, we have an identity element of $(\prod_{i\in I}P_i)^\times$.

Now let us show that every element $(x_i)_{i\in I}\in(\prod_{i\in I}P_i)^\times$ has an inverse. This is indeed true since if $(x_i)_{i\in I}\in (\prod_{i\in I}P_i)^\times$. Then we have that for each $i\in I$ that $x_i\in P_i^\times$. Thus, as each $P_i$ is a pasture, we have that for each $x_i\in P_i^\times$ there is an element $x_i\inv\in P_i$ such that $x_ix_i\inv=1_i$. I claim that $(x_i\inv)_{i\in I}$ is the inverse of $(x_i)_{i\in I}\in (\prod_{i\in I}P_i)\times$. This is clear since $(x_i)_{i\in I}(x_i\inv)_{i\in I}=(x_ix_i\inv)_{i\in I}=(1_i)_{i\in I}$. Furthermore, we note that $(x_i\inv)_{i\in I}\in (\prod_{i\in I}P_i)^\times$. This is clear since for each $i\in I$, we have that $(x_i\inv)\in P_i^\times$, so $(x_i\inv)_{i\in I}\in \prod_{i\in I}P_i^\times=(\prod_{i\in I}P_i)^\times$. This concludes the proof that $(\prod_{i\in I}P_i)^\times$ is an abelian group. 

Now we make a note that the involution of $\prod_{i\in I}P_i$ is an involution, and $-(x_i)_{i\in I}$ is an element of $\prod_{i\in I}P_i$. By definition the involution fixes $0$. Thus, for $(x_i)_{i\in I}\in (\prod_{i\in I}P_i)^\times$. Then we have that $(-x_i)_{i\in I}\in (\prod_{i\in I}P_i)^\times$ as for each $i\in I$ we have that $-x_i\in P_i^\times$. Thus, the involution will be well-defined. 

Now let us show that $N_{\prod_{i\in I}P_i}$ satisfies the properties of a nullset. Say that we have that $(a_i)_{i\in I}+(b_i)_{i\in I}+(c_i)_{i\in I}=0$ in $\prod_{i\in I}P_i$. Thus, we have that for each $i\in I$ that $a_i+b_i+c_i=0$ in $P_i$. Thus, if we permute the order of addition, since each $P_i$ is a pasture, we know that any permutation of $a_i+b_i+c_i$ will also add to $0$ in $P_i$. Hence, we will have that any permutation of $(a_i)_{i\in I}+(b_i)_{i\in I}+(c_i)_{i\in I}$ also sums to $0$. That is $N_{\prod_{i\in I}P_i}$ satisfies property (1) of being a nullset. 

Now say to show property (2) of being a nullset is satisfied. Say that $(a_i)_{i\in I}+(b_i)_{i\in I}+(c_i)_{i\in I}=0$ in $\prod_{i\in I}P_i$, and $(d_i)_{i\in I}\in \prod_{i\in I}P_i$. let us show that $(d_i)_{i\in I}(a_i)_{i\in I}+(d_i)_{i\in I}(b_i)_{i\in I}+(d_i)_{i\in I}(c_i)_{i\in I}=(d_ia_i)_{i\in I}+(d_ib_i)_{i\in I}+(d_ic_i)_{i\in I}=0$ in $\prod_{i\in I}P_i$. This corresponds to showing that for every $i\in I$ that we have that $d_ia_i+d_ib_i+d_ic_i=0$ in $P_i$. This is true. Since we know that $a_i+b_i+c_i=0$ in $P_i$ for every $i\in I$, and each $P_i$ is a pasture. We then have that $d_ia_i+d_ib_i+d_ic_i=0$ in $P_i$. Thus, we have that $N_{\prod_{i\in I}P_i}$ satisfies property (2) of being a nullset.

Finally, let us show that $N_{\prod_{i\in I}P_i}$ satisfies property (3) of being a nullset. That is say that we have that $(a_i)_{i\in I}+(b_i)_{i\in I}+(0)_{i\in I}=0$ in $\prod_{i\in I}P_i$, then we have that $(a_i)_{i\in I}=-(b_i)_{i\in I}$. This is equivalent to saying that $a_i=-b_i$ for every $i\in I$. However, we have that $a_i+b_i+0=0$ for each $i\in I$. Thus as each $P_i$ is a pasture, we have that $a_i=-b_i$ for each $i\in I$ that is $N_{\prod_{i\in I}P_i}$ satisfies property (3) of being a nullset. This concludes the proof that $\prod_{i\in I}P_i$ is a pasture. 
\end{proof}

\begin{lemma}
For each $j\in I$ the map $\pi_j:\prod_{i\in I}P_i\rightarrow P_j$ is a morphism of pastures. 
\end{lemma}

\begin{proof}
Say that $j\in I$, and consider the map $\pi_j:\prod_{i\in I}P_i\rightarrow P_j$ given by $\pi_j((x_i)_{i\in I})=x_j$. Let us show that this a morphism of pastures. We first note that $\pi_j((0)_{i\in I})=0$ and $\pi_j((1_i)_{i\in I})=1_j$. Thus, we have that $\pi_j$ preserves the zero and multiplicative identity. 

Now let us show that $\pi_j:(\prod_{i\in I}P_i)^\times\rightarrow P_j^\times$ is a homomorphism of abelian groups. This is clear since we have that $\pi_j((x_i)_{i\in I}(y_i)_{i\in I})=\pi_j((x_iy_i)_{i\in I})=x_jy_j=\pi_j((x_i)_{i\in I})\pi_j((y_i)_{i\in I})$. Thus, we have that $\pi_j$ restricts to a homomorphism of abelian groups. 

Now let us show that $\pi_j$ respects the involution. This is clear since $\pi_j(-(x_i)_{i\in I})=\pi_j((-x_i)_{i\in I})=-x_j=-\pi_j((x_i)_{i\in I})$.

Finally, let us show that $\pi_j$ preserves the additive relations. Say that $(a_i)_{i\in I}+(b_i)_{i\in I}+(c_i)_{i\in I}=0$ in $\prod_{i\in I}P_i$, then we know that $a_i+b_i+c_i=0$ in $P_i$ for each $i\in I$. Thus, we have that $\pi_j((a_i)_{i\in I})+\pi_j((b_i)_{i\in I})+\pi_j((c_i)_{i\in I})=a_j+b_j+c_j=0$ in $P_j$ that is $\pi_j$ preserves the additive relations. Hence, we have that $\pi_j$ is a morphism of pastures. 
\end{proof}

\begin{lemma}
Say that $P'$ is a pasture, and for each $i\in I$ there is a morphism $f_i:P'\rightarrow P_i$, then there is a unique morphism $f_*:P'\rightarrow \prod_{i\in I}P_i$ such that $\pi_j\circ f_*=f_j$ for each $j\in I$. 
\end{lemma}

\begin{proof}
We shall first define what the map $f_*$ and show that it is a morphism, then we shall show for each $j\in I$ that $\pi_j\circ f_*=f_j$. Lastly, we will show such a morphism $f_*$ must be unique. 

We define $f_*:P'\rightarrow \prod_{i\in I}P_i$ as $f_*(x)=(f_i(x))_{i\in I}$, we note that $f_*(0)=0$, and $f_*(1)=(f_i(1))_{i\in I}=(1_i)_{i\in I}$, so we have that $f_*$ preserves the zero and multiplicative identity. Now let us show that $f_*$ will be an abelian group homomorphism of $(P')^\times$ and $(\prod_{i\in I}P_i)^\times$. Thus, let $x\in (P')^\times$. Then we have that $f_i(x)\in P_i^\times$ for each $i\in I$. Hence, we have that $f_*(x)=(f_i(x))_{i\in I}\in (\prod_{i\in I}P_i)^\times$. Furthermore, we have that $f_*(xy)=(f_i(xy))_{i\in I}=(f_i(x)f_i(y))_{i\in I}=(f_i(x))_{i\in I}(f_i(y))_{i\in I}=f_*(x)f_*(y)$. Now we have that $f_*$ will preserve the involution since $f_*(-x)=(f_i(-x))_{i\in I}=(-f_i(x))_{i\in I}=-(f_i(x))_{i\in I}=-f_*(x)$. 

Now consider $a,b,c\in P'$ such that $a+b+c=0$ in $P'$. Now we wish to show that $f_*(a)+f_*(b)+f_*(c)=(f_i(a))_{i\in I}+(f_i(b))_{i\in I}+(f_i(c))_{i\in I}=0$. That is we wish to show that $f_i(a)+f_i(b)+f_i(c)=0$ in $P_i$ for each $i\in I$. However, since each $f_i$ is a morphism of pastures, and we have that $a+b+c=0$ in $P'$, then $f_i(a)+f_i(b)+f_i(c)=0$ in $P_i$ for every $i\in I$. That is we are have that $f_*$ will preserve the additive relations, and is indeed a morphism of pastures.

Now it let us show that $\pi_j\circ f_*=f_j$ for every $j\in I$. This is clear as we have that $\pi_j(f_*(x))=\pi_j((f_i(x)_{i\in I})=f_j(x)$.

Finally, let us show that if we have some other morphism say $g_*:P'\rightarrow \prod_{i\in I}P_i$ such that $\pi_j\circ g_*=f_j$ for every $j\in I$. Let us show that $g_*=f_*$. To do this, say that $x\in P'$, let us show that $g_*(x)=f_*(x)$. We know that $g_*(x)=((y_i)_{i\in I})$, and by definition we have that $f_*(x)=((f_i(x)_{i\in I})$. To show that these functions are the same, we will show for every $j\in I$ that $y_j=f_j(x)$. This follows from the fact that $\pi_j\circ g_*=f_j$. Thus, we have that $\pi_j(g_*(x))=y_j=f_j(x)$, so $g_*=f_*$ and $f_*$ is unique. 
\end{proof}

We summarize our findings in the following theorem.

\begin{theorem}
The category of pastures has arbitrary products. 
\end{theorem}

Now using Theorems 2.5, 4.4, and 6.5, we have fibered products, equalizers, and arbitrary products in the category of pastures. Thus, it is a well-known theorem  that a category having these properties will have arbitrary limits of diagrams indexed by a small category. We state this as the following theorem (see for example section 2 of chapter 5 of \cite{mac2013categories}). 

\begin{theorem}
The category of pastures has limits of diagrams indexed by any small category. 
\end{theorem}

\section{Arbitrary Coproducts}

In this section, we shall show that arbitrary coproducts exist in the category of pastures. Thus, if $I$ is some indexing set, and for each $i\in I$ we have some pasture $P_i$. Then the coproduct of the $P_i$ is some pasture $\otimes_{i\in I}P_i$ along with maps $i_j: P_j\rightarrow \otimes_{i\in I} P_i$ for $j\in I$. Such that for any pasture $P'$ with maps $f_i:P_i\rightarrow P'$ for each $i\in I$, then there is a unique morphism $g:\otimes_{i\in I} P_i\rightarrow P'$ such that $g\circ i_j=f_j$ for each $j\in I$. 

\begin{defn}
Given some indexing set $I$, and for each $i\in I$ we have a pasture $P_i$, then we have that the coproduct $\otimes_{i\in I} P_i$ is a pasture. As a set, we have that $\otimes_{i\in I} P_i=\{0\}\cup(\oplus_{i\in I}(P_i^\times)/\sim)$, we will say that $0=[(0)_{i\in I}]$, and all other elements will look like $[(x_i)_{i\in I}]$ such that all $x_i=1_i$ except for finitely many $x_i$. Where we have that $\sim$ is the equivalence relation given by $(x_i)_{i\in I}\sim (y_i)_{i\in I}$ if and only if there is an even number of indices say $i_1,i_2,...i_{2n}\in I$ such that for all other indices $j$ we have that $x_j=y_j$ and for these indices we have that $x_{i_k}=-y_{i_k}$ for $k=1,..,2n$. We now will define the multiplication on $\otimes_{i\in I} P_i$ as $[(x_i)_{i\in I}][(y_i)_{i\in I}]=[(x_iy_i)_{i\in I}]$. Now the involution will be given as follows $-[(x_i)_{i\in I}]=[(y_i)_{i\in I}]$ where $(y_i)_{i\in I}$ is given by picking some $j\in I$ such that $x_j=-y_j$ and $x_i=y_i$ for all $i\in I$ with $i\neq j$. Finally, we define the additive relations (for nonzero elements) as follows $[(x_i)_{i\in I}]+[(y_i)_{i\in I}]+[(z_i)_{i\in I}]=0$ if and only if there are $[(x_i')_{i\in I}]\sim [(x_i)_{i\in I}]$, $[(y_i')_{i\in I}]\sim [(y_i)_{i\in I}]$, and $[(z_i')_{i\in I}]\sim [(z_i)_{i\in I}]$ such that there is some index $j\in I$ such that $x'_j+y'_j+z'_j=0$ in $P_j$, and for each $i\neq j$ we have that  $x_i'=y_i'=z_i'$. Furthermore, the additive relations with $[0]$ are given by permutations of elements of the form $[(x_i)_{i\in I}]+[(y_i)_{i\in I}]+[(0)_{i\in I}]=0$ if and only if $ [(x_i)_{i\in I}]=-[(y_i)_{i\in I}]$. Then for each $j\in I$ we define the maps $i_j:P_j\rightarrow \otimes_{i\in I} P_i$ by $i_j(0)=[0]$ and  $i_j(x)=[(y_i)_{i\in I}]$ where $y_i=1_i$ for all $i\in I$ such that $i\neq j$, and $y_j=x$. 
\end{defn}

We will begin by showing that $\sim$ is actually an equivalence relation. Then we shall show that $\otimes_{i\in I}P_i$ is a pasture, and for $j\in I$ we have that $i_j:P_j\rightarrow \otimes_{i\in I}P_i$ is a morphism of pastures. Lastly, we will show that $\otimes_{i\in I}P_i$ satisfies the universal property of being a coproduct of the $P_i$.

\begin{lemma}
There is an equivalence relation on $\oplus_{i\in I}P_i^\times$ given by $(x_i)_{i\in I}\sim (y_i)_{i\in I}$ if and only if there is an even number of indices say $i_1,i_2,...i_{2n}\in I$ such that for all other indices $j$ we have that $x_j=y_j$ and for these indices we have that $x_{i_k}=-y_{i_k}$ for $k=1,..,2n$.
\end{lemma}

\begin{proof}
It is fairly clear that $\sim$ is reflexive since for $(x_i)_{i\in I}\in \oplus_{i\in I}P_i^\times$ we have that for each $i\in I$ that $x_i=x_i$, so $\sim$ is reflexive (just taking the empty set to be the indices which we involute). 

Now to see that $\sim$ is symmetric, if we have that $(x_i)_{i\in I}\sim (y_i)_{i\in I}$, then we have that we have that there are an even number of indices say $i_1, i_2,...,i_{2n}\in I$ such that $x_{i_k}=-y_{i_k}$ for $k=1,...,2n$, and $x_j=y_j$ for all other indices. However, this is equivalent to $y_{i_k}=-x_{i_k}$ for $i=1,...,2n$ and $y_j=x_j$ for all other indices which implies that $(y_i)_{i\in I}\sim (x_i)_{i\in I}$, so $\sim$ is symmetric. 

Finally, let us show that $\sim$ is transitive. Say that $(x_i)_{i\in I}\sim (y_i)_{i\in I}$ and $(y_i)_{i\in I}\sim (z_i)_{i\in I}$. Thus, we have that there are an even number of indices say $i_1,i_2,...,i_{2n}\in I$ such that $x_{i_k}=-y_{i_k}$ for $k=1,..,2n$ and $x_j=y_j$ for all other indices. Similarly, we have that there are an even number of indices say $j_1,j_2,...,j_{2m}\in I$ such that $y_{j_l}=-z_{j_l}$ for all $l=1,...,2m$ and $y_k=z_k$ for all other $k\in I$. Now if we have some index $i\neq i_k,j_l$ for any $k=1,...,2n$ and $l=1,...,2m$ we have that $x_i=y_i$ and $y_i=z_i$, so $x_i=z_i$. Now if we consider the remaining indices, then we shall consider two cases depending on the parity of the number of indices $i_k$ and $j_l$ which are the same. The first case is that if we have that an odd number of indices $i_k$ and $j_l$ are the same. In this case let us say that we have $k_1,...,k_{2p+1}$ are indices both of the form $i_k$ and $j_l$, so we have that $x_{k_q}=-y_{k_q}=z_{k_q}$ for $q=1,2,...,2p+1$. Then we have that there will be an odd number of the $i_k$ which are not $k_q$, and an odd number of $j_l$ which are not $k_q$, and of these we will have that $x_{i_k}=-z_{i_k}$ and $x_{j_l}=-z_{j_l}$, and there are an even number of such indices, so we have that $(x_i)_{i\in I}\sim (z_i)_{i\in I}$. Similarly, in the case that we have $k_1,...,k_{2p}$ are the number of indices both of the form $i_k$ and $j_l$, then for these indices we have that $x_{k_q}=-y_{k_q}=z_{k_q}$ for $q=1,...,2p$. Then we have that there will be an even number of the $i_k$ which are not $k_q$ and an even number of the $j_;$ which are nor $k_q$, and of these we have that $x_{i_k}=-z_{i_k}$ and $x_{j_l}=-z_{j_l}$, and of there are an even number of such indices, so we have that $(x_i)_{i\in I}\sim (z_i)_{i\in I}$. Thus, we have that $\sim$ is transitive and an equivalence relation. 
\end{proof}

\begin{lemma}
The arbitrary coproduct $\otimes_{i\in I}P_i$ is indeed a pasture, and for $j\in I$ we have that the map $i_j:P_j\rightarrow \otimes_{i\in I}P_i$ is a morphism of pastures.  
\end{lemma}

\begin{proof}
We begin by showing that multiplication is well-defined. To see this say that we have that we have $(x_i)_{i\in I}\sim (x_i')_{i\in I}$ and $(y_i)_{i\in I}\sim (y_i')_{i\in I}$. Say that $i_1,...,i_{2m}$ are the indices such that $x_{i_k}=-x'_{i_k}$ for $k=1,...,2m$. Similarly, let $j_1,...,j_{2n}$ be the indices such that $y_{j_k}=-y'_{j_k}$ for $j=1,...,2n$. We note for any index not equal to any of the $i_k$, $j_k$ we will have that $x_iy_i=x_i'y_i'$. Furthermore, of the indices $i_1,...,i_{2m}$ and $j_1,...,j_{2n}$ if there is overlap of indices say $k_1,...,k_l$, then we will have that $x_{k_i}y_{k_i}=x'_{k_i}y'_{k_i}$ while the remaining indices will be those with $x_iy_i=-x'_iy'_i$. And as in the transitivity of $\sim$ argument, we will have that there will be an even number such indices that is we will have that $(x_i)_{i\in I}(y_i)_{i\in I}=(x'_i)_{i\in I}(y'_i)_{i\in I}$. That is we have that multiplication is well-defined. 

Now we note that $[0]$ is an absorbing element under multiplication. Now since we have that multiplication is well-defined, it is clear that it will be a closed operation. The identity element will be $[(1_i)_{i\in I}]$ since we have that given any $[(x_i)_{i\in I}]\in \otimes_{i\in I}P_i$, we have that $[(1_i)_{i\in I}][(x_i)_{i\in I}]=[(x_i)_{i\in I}]$, so we have an identity element. Now for inverses, say that we have $[(x_i)_{i\in I}]$, and we have that $i_1,...,i_n$ are the indices such that for $j\neq i_k$ we have that $x_j=1$. Then we will have that $[(x_i)_{i\in I}]\inv=[(x_i\inv)_{i\in I}]$ since $1_j\inv=1_j$ we have that this will indeed be an element which is clearly a multiplicative inverse. Thus, we have that $(\otimes_{i\in I}P_i)^\times$ will be an abelian group. 

Now to see that the involution is indeed an involution. We note that given $[(x_i)_{i\in I}]$, then we have that $-[(x_i)_{i\in I}]=[(y_i)_{i\in I}]$ chooses some index say $j$ and sets $y_j=-x_j$ and $y_i=x_i$ for all $i\neq j$. We note that involuting again gives us $--[(x_i)_{i\in I}]=[(z_i)_{i\in I}]$ twice corresponds to picking two indices say $j,k$ and setting $z_j=-x_j$, $z_k=-x_k$, and $z_i=x_i$ for all $i\neq j,k$. Thus, we see that $--[(x_i)_{i\in I}]=[(x_i)_{i\in I}]$ by the definition of the equivalence relation. Thus, the involution is well-defined. 

Now we show that the additive relations satisfy the axioms of the null set. We note that it is clear from the additive relations that they will be invariant under permutations. Similarly, we note that the relations involving $0$ make it such that property (3) of the null set is satisfied. Thus, for property (2), say that we have that (non-zero) $[(x_i)_{i\in I}]+[(y_i)_{i\in I}]+[(z_i)_{i\in I}]=0$ where we have chosen these representatives such that there is an index $j$ such that $x_j+y_j+z_j=0$ and $x_i=y_i=z_i$ for $i\neq j$. Then for any $[(w_i)_{i\in I}]$ if we consider 
\[
[(x_i)_{i\in I}][(w_i)_{i\in I}]+[(y_i)_{i\in I}][(w_i)_{i\in I}]+[(z_i)_{i\in I}][(w_i)_{i\in I}]=0
\]
Since the $j$th index will be $x_jw_j+y_jw_j+z_jw_j=0$, and for $i\neq j$ we have that $x_iw_i=y_iw_i=z_iw_i$. Thus, we have that property of the null set (2) is satisfied. we note in the case where there is a zero element in the sum, a similar proof holds. Thus, the additive relations are satisfied, and we have that $\otimes_{i\in I}P_i$ is a pasture. 

Now to see that for each $j\in I$ we have that the function $i_j:P_j\rightarrow \otimes_{i\in I}P_i$ is a morphism of pastures. First it is clear that $i_j(0)=0$ and $i_j(1)=[(1_i)_{i\in I}]$. To see that $i_j$ respects multiplication we note that for $x,y\in P_j$, we have that $i_j(xy)=[(z_i)_{i\in I}]$ where $z_j=xy$ and $z_i=1_1$ for every $i\neq j$. Furthermore, we note that $i_j(x)=[(x_i)_{i\in I}]$ where $x_j=x$ and $x_i=1_i$ for $i\neq j$. Similarly, $i_j(y)=[(y_i)_{i\in I}]$ where $y_j=y$ and $y_i=1_i$ for $i\neq j$. From this it is clear that $i_j(xy)=i_j(x)i_j(y)$. Thus, we have that $i_j$ respects multiplication. 

Now to see that $i_j$ respects the additive relations, we note that if $x+y+z=0$ in $P_j$, then we have that $i_j(x)=[(x_i)_{i\in I}]$ where $x_j=x$ and $x_i=1_i$ for $i\neq j$, and similarly for $i_j(y)$ and $i_j(z)$. Thus, if we consider $i_j(x)+i_j(y)+i_j(z)=0$ since we have for all $i\neq j$ we have that $x_i=y_i=z_i=1_i$, and for the index $j$ we have that $x_j+y_j+z_j=x+y+z=0$. Thus, we have that $i_j$ respects the additive relations, and is a morphism of pastures. 
\end{proof}

\begin{lemma}
If $P$ is a pasture and for each $i\in I$ we have that $f_i:P_i\rightarrow P$ is a morphism of pastures, then there exists a unique map $g:\otimes_{i\in I}P_i\rightarrow P_i$ such that for each $j\in I$ we have that $f_j=g\circ i_j$ that is $\otimes_{i\in I}P_i$ satisfies the universal property of the coproduct. 
\end{lemma}

\begin{proof}
We shall begin by defining $g:\otimes_{i\in I}P_i\rightarrow P$. For $[(x_i)_{i\in I}]\in \otimes_{i\in I}P_i$, we know that all $x_i=1_i$ except for finitely many say for $i_1,...i_n\in I$ we have that $x_{i_j}\neq 1_{i_j}$ and for the remaining indices $x_i=1_i$. Thus, we define $g([(x_i)_{i\in I}])=f_{i_1}(x_{i_1})f_{i_2}(x_{i_2})...f_{i_n}(x_{i_n})$. We make the note that if we also include indices $j\in I$ such that $x_j=1_j$, then as $f_j(1_j)=1$ adding this to the product will not change it, so it is helpful to allow some of the indices to be $1$. In the case of where $x_i=1_i$ for each $i\in I$ we define $g([(1_i)]_{i\in I})=1$, and we also define $g(0)=0$. 

We first will show that $g$ is well-defined. That is if we have that $[(x_i)_{i\in I}]\sim [(y_i)_{i\in I}]$, then $g([(x_i)_{i\in I}])=g([(y_i)_{i\in I}])$.

Let $i_1,...,i_m$ be the indices where either $x_{i_l}\neq 1$ or $y_{i_l}\neq 1$. Furthermore, since $[(x_i)_{i\in I}]\sim [(y_i)_{i\in I}]$, we have an even number of indices just order them say $i_1,...,i_{2n}$ are such that $x_{i_k}=-y_{i_k}$ for $k=1,...,2n$. Then we have that 
\[
g([(x_i)_{i\in I}])=(\prod_{k=1}^{2n}f_{i_k}(x_{i_k}))(\prod_{k=2n+1}^mf_{ik}(x_{i_k}))=(\prod_{k=1}^{2n}f_{i_k}(-y_{i_k}))(\prod_{k=2n+1}^mf_{ik}(y_{i_k}))=f([(y_i)_{i\in I}])
\]
Thus, we have that $g$ is a well-defined function. 

We now show that $g$ is a morphism of pastures. We have by definition that $g([(0)_{i\in I}])=0$, and $g([(1_i)_{i\in I}])=1$. To see that $g$ is multiplicative say that we have $[(x_i)_{i\in I}],[(y_i)_{i\in I}]\in \otimes_{i\in I}P_i$, and let $i_1,...,i_n$ be indices such that either $x_{i_k}\neq 1_{i_k}$ or $y_{i_k}\neq 1_{i_k}$ for $k=1,...,n$, and $x_j=y_j=1_j$ for $j\neq i_k$. Then we have that 
\[
g([(x_i)_{i\in I}][(y_i)_{i\in I}])=g([(x_iy_i)_{i\in I}])=f_{i_1}(x_{i_1}y_{i_1})...f_{i_n}(x_{i_n}y_{i_n})
\]
\[
=f_{i_1}(x_{i_1})f_{i_1}(y_{i_1})...f_{i_n}(x_{i_n})f_{i_n}(y_{i_n})=g([(x_i)_{i\in I}])g([(y_i)_{i\in I}])
\]
Thus, we have that $g$ preserves the multiplication. Now to see that $g$ preserves the involution say that we have $[(x_i)_{i\in I}]\in \otimes_{i\in I}P_i$, and say that $i_1,...,i_n$ are the indices such that $x_{i_k}\neq 1_{i_k}$. Then we have that $-[(x_i)_{i\in I}]=[(y_i)_{i\in I}]$ where say $x_{i_1}=-y_{i_1}$ and $x_j=y_j$ for all other $j\neq i_1$. Then we have that 
\[
-g([(x_i)_{i\in I}])=-f_{i_1}(x_{i_1})...f_{i_n}(x_{i_n})=f_{i_1}(-x_{i_1})...f_{i_n}(x_{i_n})=f_{i_1}(y_{i_1})...f_{i_n}(y_{i_n})=g([(y_i)_{i\in I}])
\]
Thus, we have that $g$ preserves the involution. 

Finally, to see that the additive relations are preserved. That is let us say that we have $[(x_i)_{i\in I}],[(y_i)_{i\in I}],[(z_i)_{i\in I}]\in\otimes_{i\in I}P_i$ such that $[(x_i)_{i\in I}]+[(y_i)_{i\in I}]+[(z_i)_{i\in I}]=0$. Now we wish to show that $g([(x_i)_{i\in I}])+g([(y_i)_{i\in I}])+g([(z_i)_{i\in I}])=0$. To see this we have that there is some index $j$ such that $x_j+y_j+z_j=0$ in $P_j$, so we have that $f_j(x_j)+f_j(y_j)+f_j(z_j)=0$ in $P'$. Let us first make the assumption that each of these values are nonzero, then for all other indices $i\neq j$ we have that $x_i=y_i=z_i$. Now if we consider the values say $j,i_1,i_2,...,i_n$ in which we have that each of the $x_{i_k},y_{i_k},z_{i_k}\neq 1_{i_k}$ for $k=1,..,n$. Then we have that 
\[
g([(x_i)_{i\in I}])+g([(y_i)_{i\in I}])+g([(z_i)_{i\in I}])
\]
\[
=f_j(x_j)f_{i_1}(x_{i_1})...f_{i_n}(x_{i_n})+f_j(y_j)f_{i_1}(y_{i_1})...f_{i_n}(y_{i_n})+f_j(z_j)f_{i_1}(z_{i_1})...f_{i_n}(z_{i_n})
\]
\[
=f_j(x_j)f_{i_1}(x_{i_1})...f_{i_n}(x_{i_n})+f_j(y_j)f_{i_1}(x_{i_1})...f_{i_n}(x_{i_n})+f_j(z_j)f_{i_1}(x_{i_1})...f_{i_n}(x_{i_n})=0
\]
In the case when we have that one of the elements is $0$, without loss of generality say we have $[(z_i)_{i\in I}]=[(0)_{i\in I}]$, then we have that $[(x_i)_{i\in I}]=-[(y_i)_{i\in I}]$. Then we have that 
\[
g([(x_i)_{i\in I}])+g([(y_i)_{i\in I}])+g([(0)_{i\in I}])=g([(x_i)_{i\in I}])+-g([(x_i)_{i\in I}])+g([(0)_{i\in I}])=0
\]
Thus, we have that $g$ preserves the additive relations and that $g$ is indeed a morphism of pastures. 

Now let us show that for $j\in I$ we have that $f_j=g\circ i_j$. This is clear since we have that for $x\in P_j^\times$ we have that $g(i_j(x))=g([(y_i)_{i\in I}])$ where $y_j=x$ and $y_i=1_i$ for all $i\neq j$. Thus, we have that $g([(y_i)_{i\in I}])=f_j(x)$, and we have that $g(i_j(0))=g(0)=0=f_j(0)$. Hence, we have that $f_j=g\circ i_j$ for each $j\in I$. 

Now let us show that $g$ is unique. Say that we have some other function say $h$ such that $f_j=h\circ i_j$ for each $j\in I$. let us show that $h=g$. Say that we have some element $[(x_i)_{i\in I}]\in\otimes_{i\in I}P_i$, and say that $i_1,...,i_n$ are the indices such that $x_j=1_j$ for $j\neq i_1,...,1_n$. Now since we have that $f_j=h\circ i_j$ we have that if $[(x_i)_{i\in I}]\in \otimes_{i\in I}P_i$ such that $x_j=x_j$ and $x_i=1_i$ for all $i\neq j$, then we have that $h([(x_i)_{i\in I}])=f_j(x_j)$. Now let us denote $[(y_j^{i_k})_{j\in I}]$ be such that $y_{i_k}=x_{i_k}$ and $y_j=1_j$ for all $j\neq i_k$. Then we have that $[(x_i)_{i\in I}]=\prod_{k=1}^n[(y_j^{i_k})_{j\in I}]$; thus, we have that 
\[
h([(x_i)_{i\in I}])=h(\prod_{k=1}^n[(y_j^{i_k})_{j\in I}])=\prod_{k=1}^n h([(y_j^{i_k})_{j\in I}])=\prod_{k=1}^nf_{i_k}(x_{i_k})=g([(x_i)_{i\in I}])
\]
Thus, we have that $g=h$, so $g$ must be unique. 
\end{proof}

We summarize our results as the following theorem.

\begin{theorem}
The category of pastures has arbitrary coproducts.
\end{theorem}

Now using Theorems 3.6, 5.5, and 7.5, we have fibered products, equalizers, and arbitrary products in the category of pastures. Thus, by the dual theorem we used to show we had limits, we have that the category of pastures  have arbitrary colimits of diagrams indexed by a small category. We state this as the following theorem. 

\begin{theorem}
The category of pastures has colimits of diagrams indexed by any small category. 
\end{theorem}

\bibliographystyle{alpha}
\bibliography{Ref}

\begin{thebibliography}{ML13}

\bibitem[BL18]{baker2018moduli}
Matthew Baker and Oliver Lorscheid.
\newblock The moduli space of matroids.
\newblock {\em arXiv preprint arXiv:1809.03542}, 2018.

\bibitem[BL20]{baker2020foundations}
Matthew Baker and Oliver Lorscheid.
\newblock Foundations of matroids i: Matroids without large uniform minors.
\newblock {\em arXiv preprint arXiv:2008.00014}, 2020.

\bibitem[BL21]{baker2021descartes}
Matthew Baker and Oliver Lorscheid.
\newblock Descartes' rule of signs, newton polygons, and polynomials over
  hyperfields.
\newblock {\em Journal of Algebra}, 569:416--441, 2021.

\bibitem[ML13]{mac2013categories}
Saunders Mac~Lane.
\newblock {\em Categories for the working mathematician}, volume~5.
\newblock Springer Science \& Business Media, 2013.

\end{thebibliography}

\end{document}